\documentclass[preprint]{article}%
\usepackage{graphicx}
\usepackage{amsmath,amsxtra,amssymb,latexsym, amscd,amsthm}
\usepackage[mathscr]{eucal}
\newtheorem{thm}{Theorem}[section]

\newtheorem{lem}{Lemma}[section]
\newtheorem{prop}{Proposition}[section]
\theoremstyle{definition}

\theoremstyle{remark}
\newtheorem{rem}{Remark}[section]
\numberwithin{equation}{section}

\begin{document}

\title{Density estimates for the exponential functionals of fractional Brownian motion}
\author{Nguyen Tien Dung\thanks{Department of Mathematics, VNU University of Science, Vietnam National University, Hanoi, 334 Nguyen
Trai, Thanh Xuan, Hanoi, 084 Vietnam.}\,\,\footnote{Corresponding author. Email: dung@hus.edu.vn}
\and Nguyen Thu Hang\thanks{Department of Mathematics, Hanoi University of Mining and Geology, 18 Pho Vien, Bac Tu Liem, Hanoi, 084 Vietnam.}
\and Pham Thi Phuong Thuy\thanks{The faculty of Basic Sciences, Vietnam Air Defence and Air Force Academy, Son Tay, Ha Noi, 084 Vietnam.}
}

\date{September 22, 2021}          

\maketitle
\begin{abstract} In this note, we investigate the density of the exponential functional of the fractional Brownian motion. Based on the techniques of Malliavin's calculus, we provide a log-normal upper bound for the density.
\end{abstract}
\section{Introduction}
Let $B^H=(B^H_t)_{t\in [0,T]}$ be a fractional Brownian motion (fBm) with Hurst index $H\in(0,1).$ We recall that $B^H$ is a centered Gaussian process with covariance function
$$R_H(t,s):=E[B^H_t B^H_s]=\frac{1}{2}(t^{2H}+s^{2H}-|t-s|^{2H}),\,\,0\leq s,t\leq T.$$
We consider the exponential functional of the form
\begin{equation}\label{yu1}
F=\int_0^T e^{as+\sigma B^H_s}ds,
\end{equation}
where $T>0,a\in \mathbb{R}$ and $\sigma>0$ are constants. It is known that this functional plays an important role in several domains.  The special case, where $H=\frac{1}{2},$ has been well studied and a lot of fruitful properties of $F$ can be founded in the literature, see e.g. \cite{Matsumoto2005a,Matsumoto2005b,Yor2001}.  However, to the best our knowledge, the deep properties of $F$ for $H\neq \frac{1}{2}$ are scarce. In a recent paper \cite{DungPG2019}, we have proved the Lipschitz continuity of the cumulative distribution function of $F$ with respect to the Hurst index $H.$ The aim of the present paper is to investigate the density of $F.$ Unlike the case $H=\frac{1}{2},$ it is not easy to find the density of $F$ explicitly for $H\neq \frac{1}{2}$ and hence, our work will focus on providing the estimates for the density function. It should be noted that, in the last years,  the density estimates for random variables related to fBm has been extensively studied, see e.g. \cite{Besalu2016,DungPG2016,Liu2017,Tan2020} and references therein.

The rest of this article is organized as follows. In Section \ref{uufk}, we briefly recall some of the relevant elements of the Malliavin calculus and two general estimates for densities. Our main results are then stated and proved in Section 3. Our Theorems \ref{kl1} and \ref{kl2} point out that the density of $F$ is bounded from above by log-normal densities.
\section{Preliminaries}\label{uufk}
In the whole paper, we assume $H>\frac{1}{2}.$ Under this assumption, fBm admits the Volterra representation
\begin{equation}\label{9jdm3}
B^H_t=\int_0^t K(t,s)dB_s,
\end{equation}
where $(B_t)_{t\in [0,T]}$ is a standard Brownian motion and  for some normalizing constant $c_H,$ the kernel  $K$ is given by
$$ K(t,s) = c_{H}s^{1/2 -H}\int_{s}^{t}(u-s)^{H-\frac{3}{2}}u^{H-\frac{1}{2}}du,\,\,0<s\leq t\leq T.$$
Let us recall some elements of Malliavin calculus with respect to Brownian motion $B,$ where $B$ is used to present $B^H_t$ as in (\ref{9jdm3}).  We suppose that $(B_t)_{t\in [0,T]}$ is defined on
a complete probability space $(\Omega,\mathcal{F},\mathbb{F},P)$, where $\mathbb{F}=(\mathcal{F}_t)_{t\in [0,T]}$ is a natural filtration generated by
the Brownian motion $B.$ For $h\in L^2[0,T],$ we denote by $B(h)$ the Wiener integral
$$B(h)=\int_0^T h(t)dB_t.$$
Let $\mathcal{S}$ denote the dense subset of $L^2(\Omega, \mathcal{F},P)$ consisting of smooth random variables of the form
\begin{equation}\label{ro}
F=f(B(h_1),...,B(h_n)),
\end{equation}
where $n\in \mathbb{N}, f\in C_b^\infty(\mathbb{R}^n),h_1,...,h_n\in L^2[0,T].$ If $F$ has the form (\ref{ro}), we define its Malliavin derivative as the process $DF:=\{D_tF, t\in [0,T]\}$ given by
$$D_tF=\sum\limits_{k=1}^n \frac{\partial f}{\partial x_k}(B(h_1),...,B(h_n)) h_k(t).$$
More generally, for each $k\geq 1,$ we can define the iterated derivative operator  by setting
$$D^{k}_{t_1,...,t_k}F=D_{t_1}...D_{t_k}F.$$
For any $p,k\geq 1,$ we shall denote by $\mathbb{D}^{k,p}$ 
the closure of $\mathcal{S}$ with respect to the norm
$$\|F\|^p_{k,p}:=E|F|^p+E\bigg[\int_0^T|D_{t_1}F|^pdt_1\bigg]+...+E\bigg[\int_0^T...\int_0^T|D^{k}_{t_1,...,t_k}F|^pdt_1...dt_k\bigg].$$
A random variable $F$ is said to be Malliavin differentiable if it belongs to $\mathbb{D}^{1,2}.$ For any  $F\in \mathbb{D}^{1,2},$ the Clark-Ocone formula says that
$$F-E[F]=\int_0^TE[D_sF|\mathcal{F}_s]dB_s.$$
Moreover, any  $F,G\in \mathbb{D}^{1,2},$ we have the following covariance formula
\begin{equation}\label{lfm9}
{\rm Cov}(F,G)=E\left[\int_0^TD_sF E[D_sG|\mathcal{F}_s]ds\right].
\end{equation}
In order to obtain the density estimates for exponential functionals we need the following general results.
\begin{prop}\label{p31}
 Let $q, \alpha, \beta$ be three positive real numbers such that $\frac{1}{q}+\frac{1}{\alpha}+\frac{1}{\beta}=1$. Let $F$ be a random variable in the space $\mathbb{D}^{2,\alpha}$, such that $E[||DF||^{-2\beta}_{H}] < \infty$. Then the density $\rho_{F}(x)$ of $F$ can be estimated as follows
\begin{align}
\rho_{F}(x) \leq c_{q, \alpha, \beta}(P(F\leq x))^{1/q}
	 \times \left(E[||DF||_{H}^{-1}]+||D^{2}F||_{L^{\alpha}(\Omega; H\otimes H)}\parallel ||DF||_{H}^{-2}\parallel_{\beta}\right),\,\,x\in \mathbb{R},
\end{align}
where $c_{q, \alpha, \beta}$ is a positive constant and $H=L^2[0,T].$
\end{prop}
\begin{proof}This proposition comes from the computations on page 87 in \cite{nualartm2}.
\end{proof}
\begin{prop}\label{9hj3}Let $F\in\mathbb{D}^{2,4}$ be such that $E[F]=0.$ Define the random variable
$$\Phi_F:=\int_0^TD_sFE[D_sF|\mathcal{F}_s]ds.$$
 Assume that $\Phi_F\neq 0$ a.s. and the random variables $\frac{F}{\Phi_F}$ and $\frac{1}{\Phi_F^2}\int_0^TD_s\Phi_FE[D_sF|\mathcal{F}_s]ds$ belong to $L^2(\Omega).$ Then the law of $F$ has a continuous density given by
\begin{equation}\label{krkl2}
\rho_F(x)= \rho_F(0)\exp\left(-\int_0^x h_F(z)dz\right)\exp\left(-\int_0^x w_F(z)dz\right),\,\,\,x\in\mathrm{supp}\,\rho_F,
\end{equation}
where the functions $w_F$ and $h_F$ are defined by
$$w_F(z):=E\left[\frac{F}{\Phi_F}\big| F=z\right],\,\,\,h_F(z):=E\left[\frac{1}{\Phi_F^2}\int_0^TD_s\Phi_FE[D_sF|\mathcal{F}_s]ds\big| F=z\right].$$
\end{prop}
\begin{proof}
This proposition is Theorem 7 in our recent paper \cite{Dung2020}.
\end{proof}

\section{The main results}\label{hje3}
In this Section, we provide explicit estimates for the density $\rho_F(x)$ of the functional $F$ defined by (\ref{yu1}). Our idea is to consider the random variable $X:=\ln F-E[\ln F]$ and use the relation $\rho_F(x)=\frac{1}{x}\rho_X(\ln x-E[\ln F]),\,\,x>0,$ where $\rho_X$ denotes the density of $X.$

 We need some technical results.
\begin{prop}\label{kfl6} Consider the random variable
$X:=\ln F-E[\ln F].$  It holds that
\begin{equation}\label{kld2}
0\leq D_\theta X\leq \sigma K(T,\theta)\,\,a.s.
\end{equation}
\begin{equation}\label{kld3}
0\leq D_rD_\theta X\leq 2\sigma^2K(T,\theta)K(T,r)\,\,a.s.
\end{equation}

\end{prop}
\begin{proof}
By the chain rule for Malliavin derivatives, we have, for $0\leq r,\theta\leq T,$
\begin{equation}\label{ikf4}
D_\theta X=\frac{\sigma\int_\theta^T K(s,\theta)e^{as+\sigma B^H_s}ds}{\int_0^T e^{as+\sigma B^H_s}ds}
\end{equation}
and
$$D_rD_\theta X=\frac{\sigma^2\int_{\theta\vee r}^T K(s,\theta)K(s,r)e^{as+\sigma B^H_s}ds}{\int_0^T e^{as+\sigma B^H_s}ds}-\frac{\sigma^2\int_r^T K(s,r)e^{as+\sigma B^H_s}ds\int_\theta^T K(s,\theta)e^{as+\sigma B^H_s}ds}{\big(\int_0^T e^{as+\sigma B^H_s}ds\big)^2}.$$
Because the function $s\mapsto K(s,\theta)$ is non-decreasing for each $\theta,$ (\ref{kld2}) follows directly from (\ref{ikf4}). We also have
$$D_rD_\theta X\leq 2\sigma^2K(T,\theta)K(T,r)\,\,a.s.$$
To prove the non-negativity of the second order Malliavin derivative, we let $U$ be a random variable with the density function defined by
$$f(x)=\frac{e^{ax+\sigma B^H_x}}{\int_0^T e^{as+\sigma B^H_s}ds},\,\,0\leq x\leq T.$$
Denote by $E_U$ the expectation with respect to $U.$ We have
$$D_rD_\theta X=E_U[K(U,\theta)K(U,r)]-E_U[K(U,\theta)]E_U[K(U,r)].$$
Note that the functions $s\mapsto K(s,\theta)$ and $s\mapsto K(s,r)$ are non-decreasing. Hence, by Chebyshev's association inequality, $D_rD_\theta X\geq 0\,\,a.s.$ The proof of Proposition is complete.
\end{proof}
\begin{lem}\label{lem3.1}
Define $$M_{r}:=E\left[F|\mathcal{F}_r\right]=E\left[\int_0^Te^{as+\sigma B_s^H}ds\big|\mathcal{F}_r\right],\,\,0\leq r\leq T.$$
Then, for every $p\geq 2,$ we have
$$E\left[(\max\limits_{0\leq r\leq T}M_{r})^p\right]\leq C<\infty,$$
where $C$ is a positive constant depending on $p,T,a,\sigma$ and $H.$
\end{lem}
\begin{proof}The stochastic process $M:=(M_{r})_{0\leq r\leq T}$ is a martingale with $M_0=E[F]$ and $M_T=F.$ Hence, by Burkh\"older-David-Gundy inequality, we have
\begin{equation}\label{9iof}
E\left[(\max\limits_{0\leq r\leq T}M_{r})^p\right]\leq c_p\left(M_{0}^p+E[\langle M\rangle_T^{p/2}]\right)= c_p\left((E[F])^p+E[\langle M\rangle_T^{p/2}]\right),
\end{equation}
where $c_p$ is a positive constant. Using the Clark-Ocone formula we have
\begin{align*}
M_{T}&=EM_{T}+\int_0^TE[D_rM_{T}|\mathcal{F}_r]dB_r\\&=E[F]+\sigma\int_0^T E\left[\int_r^T K(s,r)e^{as+\sigma B^H_s}ds\big|\mathcal{F}_r\right]dB_r,
\end{align*}
which gives us
\begin{align*}
\langle M\rangle_T&=\int_0^T \sigma^2 \left(E\left[\int_r^T K(s,r)e^{as+\sigma B^H_s}ds\big|\mathcal{F}_r\right]\right)^2dr\\
&\leq \int_0^T \sigma^2 K^2(T,r) M^2_rdr\,\,a.s.
\end{align*}
Then, by H\"older inequality, we have
\begin{align}
E[\langle M\rangle_T^{p/2}]&\leq \sigma^pE\left[\bigg(\int_0^T K^\frac{2p-4}{p}(T,r)K^\frac{4}{p}(T,r)M_{r}^2dr\bigg)^{p/2}\right]\notag\\
&\leq \sigma^p\bigg(\int_0^T K^{2}(T,r)dr\bigg)^{\frac{p}{2}-1}\bigg(\int_0^T K^2(T,r)E\left[M_{r}^{p}\right]dr\bigg)\notag\\
&\leq \sigma^pT^{(p-2)H}\bigg(\int_0^T K^2(T,r)E[F^p]dr\bigg)\notag\\
&= \sigma^pT^{pH}E[F^p].\label{9iofa}
\end{align}
Here we used the fact that $\int_0^T K^2(T,r)dr=E|B^H_T|^2=T^{2H}.$ So we obtain the desired conclusion by inserting (\ref{9iofa}) into (\ref{9iof}).
\end{proof}
\begin{prop}\label{iklw} Let $X$ be as in Proposition \ref{kfl6}. We define $\Phi_X:=\int_0^TD_sXE[D_sX|\mathcal{F}_s]ds.$ Then,
$$|\Phi_X|^{-1}\in L^p(\Omega),\,\,\forall\,p\geq 1.$$
We also have
$$\left(\int_0^T |D_\theta X|^2d\theta\right)^{-1}\in L^p(\Omega),\,\,\forall\,\,p\geq 1.$$
\end{prop}
\begin{proof} It follows from (\ref{ikf4}) that
\begin{equation}\label{uuu1}
D_\theta X\geq \frac{\sigma}{T}e^{-2|a|T+\sigma \min\limits_{0\leq s\leq T} B^H_s-\sigma \max\limits_{0\leq s\leq T} B^H_s}\int_\theta^T K(s,\theta)ds\,\,a.s.
\end{equation}
On the other hand, by using the Cauchy-Schwarz  inequality, we have
$$
E[D_\theta X|F_\theta]\geq \frac{\sigma\left(E\bigg[\sqrt{\int_\theta^T K(s,\theta)e^{as+\sigma B^H_s}ds}\big|F_\theta\bigg]\right)^2}{E\bigg[\int_0^T e^{as+\sigma B^H_s}ds\big|F_\theta\bigg]}\,\,a.s.
$$
and
$$\sqrt{\int_\theta^T K(s,\theta)e^{as+\sigma B^H_s}ds}\geq \frac{\int_\theta^T K(s,\theta)\sqrt{e^{as+\sigma B^H_s}}ds}{\sqrt{\int_\theta^T K(s,\theta)ds}}\,\,a.s.$$
We therefore get
$$
E[D_\theta X|F_\theta]\geq \frac{\sigma\left(\int_\theta^T K(s,\theta)E\big[e^{as/2+\sigma B^H_s/2}\big|F_\theta\big]ds\right)^2}{\int_\theta^T K(s,\theta)dsE\bigg[\int_0^T e^{as+\sigma B^H_s}ds\big|F_\theta\bigg]}\,\,a.s.
$$
Furthermore, by Lyapunov's inequality, $$E\big[e^{as/2+\sigma B^H_s/2}\big|F_\theta \big]\geq e^{as/2+\sigma E[B^H_s|F_\theta]/2}\,\,a.s.$$
As a consequence,
\begin{equation}\label{uuu2}
E[D_\theta X|F_\theta]\geq \frac{\sigma e^{-|a|T+\sigma \min\limits_{0\leq \theta\leq s\leq T} N_{s,\theta}}\int_\theta^T K(s,\theta)ds}{\max\limits_{0\leq \theta\leq T}M_\theta}\,\,a.s.
\end{equation}
where $N_{s,\theta}:=E[B^H_s|F_\theta]$ and $M_\theta:=E\bigg[\int_0^T e^{as+\sigma B^H_s}ds\big|F_\theta\bigg].$

Combining (\ref{uuu1}) and (\ref{uuu2}) yields
$$D_\theta XE[D_\theta X|F_\theta]\geq \frac{\sigma^2}{T}e^{-3|a|T+\sigma \min\limits_{0\leq s\leq T} B^H_s-\sigma \max\limits_{0\leq s\leq T} B^H_s+\sigma \min\limits_{0\leq \theta\leq s\leq T} N_{s,\theta}} \frac{\left(\int_\theta^T K(s,\theta)ds\right)^2}{\max\limits_{0\leq \theta\leq T} M_\theta}\,\,a.s.$$
and hence,
\begin{equation}\label{ol0}
\Phi_X\geq \frac{\sigma^2}{T}e^{-3|a|T+\sigma \min\limits_{0\leq s\leq T} B^H_s-\sigma \max\limits_{0\leq s\leq T} B^H_s+\sigma \min\limits_{0\leq \theta\leq s\leq T} N_{s,\theta}} \frac{\int_0^T\left(\int_\theta^T K(s,\theta)ds\right)^2d\theta}{\max\limits_{0\leq \theta\leq T} M_\theta}\,\,a.s.
\end{equation}
We observe that
\begin{align*}
\int_0^T\left(\int_\theta^T K(s,\theta)ds\right)^2d\theta&=\int_0^T\int_\theta^T \int_\theta^TK(t,\theta)K(s,\theta)dsdtd\theta\\
&=\int_0^T\int_0^T\left( \int_0^{s\wedge t}K(t,\theta)K(s,\theta)d\theta\right)dsdt\\
&=\int_0^T\int_0^TE[B^H_tB^H_s]dsdt=\frac{T^{2H+2}}{2H+2}.
\end{align*}
This, together with (\ref{ol0}), yields
$$|\Phi_X|^{-1}\leq \frac{2H+2}{\sigma^2T^{2H+1}}e^{3|a|T+2\sigma \max\limits_{0\leq s\leq T} B^H_s+\sigma \max\limits_{0\leq \theta\leq s\leq T} N_{s,\theta}} \max\limits_{0\leq \theta\leq T} M_\theta\,\,a.s.$$
We have $(N_{s,\theta})_{0\leq \theta\leq s\leq T}$ is a Gaussian field with finite variances because $N_{s,\theta}=\int_0^\theta K_H(s,r)dB_r.$ Hence, by Fernique's theorem, there exists $\varepsilon>0$ such that  $E\left[e^{\varepsilon \max\limits_{0\leq \theta\leq s\leq T} |N_{s,\theta}|^2}\right]<\infty.$ Since $e^{p\sigma \max\limits_{0\leq \theta\leq s\leq T} N_{s,\theta}}\leq e^{\frac{p^2\sigma^2}{4\varepsilon}+\varepsilon \max\limits_{0\leq \theta\leq s\leq T} |N_{s,\theta}|^2},$ this implies that $e^{\sigma \max\limits_{0\leq \theta\leq s\leq T} N_{s,\theta}}\in L^p(\Omega)$ for any $p\geq 1.$ Similarly, we also have $e^{2\sigma \max\limits_{0\leq s\leq T} B^H_s}\in L^p(\Omega)$ for any $p\geq 1.$ So, recalling Lemma \ref{lem3.1}, we conclude that $|\Phi_X|^{-1}\in L^p(\Omega)$ for any $p\geq 1.$

We deduce from (\ref{uuu1}) that
\begin{equation}
\int_0^T|D_\theta X|^2d\theta\geq \frac{\sigma^2}{(2H+2)T^{2H}}e^{-4|a|T+2\sigma \min\limits_{0\leq s\leq T} B^H_s-2\sigma \max\limits_{0\leq s\leq T} B^H_s}\,\,a.s.
\end{equation}
Hence, we also have $\left(\int_0^T |D_\theta X|^2d\theta\right)^{-1}\in L^p(\Omega),\,\,\forall\,\,p\geq 1.$ The proof of Proposition is complete.
\end{proof}
We now are in a position to bound the density $\rho_F(x)$ of $F.$ We first use Proposition \ref{p31} to estimate the left tail of the density.
\begin{thm}\label{kl1} We have
\begin{equation}\label{jkvm1}
\rho_F(x)\leq \frac{c}{x}\exp\left({-\frac{(\ln x-E[\ln F])^2}{8\sigma^2T^{2H}}}\right),\,\,0<x\leq e^{E[\ln F]},
\end{equation}
where $c$ is a positive constant.
\end{thm}
\begin{proof} It is known from Proposition \ref{iklw} that
$$||DX||^{-2}_{H}=\left(\int_0^T |D_\theta X|^2d\theta\right)^{-1}\in L^p(\Omega),\,\,\forall\,\,p\geq 1.$$
In addition, from the estimate (\ref{kld3}), we have
$$||D^{2}X||_{L^{2}(\Omega; H\otimes H)}^2=\int_0^T\int_0^T E|D_\theta D_rX|^2d\theta dr\leq  \sigma^4\int_{0}^T \int_0^TK^2(T,\theta)K^2(T,r)d\theta dr= \sigma^4 T^{2H}<\infty.$$
The above estimates allow us to use Proposition \ref{p31} with $q=\beta=4,\alpha=2$ and we obtain
\begin{equation}\label{jkvm1a}
\rho_X(x)\leq cP(X\leq x)^{\frac{1}{4}},\,\,x\in \mathbb{R},
\end{equation}
where $c$ is a positive constant.

The remaining of the proof is to bound $P(X\leq x)$ for $x\leq 0.$ We  consider the function $\varphi(\lambda):=E[e^{-\lambda  X}],\,\,\lambda>0$ (this function is well defined because $F^{-1}\in L^p(\Omega),\,\,\forall\,\,p\geq 1$). By using repeatedly the covariance formula (\ref{lfm9}), we have
$$\sigma^2_X:={\rm Var}(X)=E[\Phi_X]$$
and
\begin{align*}
\varphi'(\lambda)&=-E[ X e^{-\lambda  X}]\\
&=\lambda E[e^{-\lambda  X}\Phi_X]\\
&=\lambda \sigma^2_XE[e^{-\lambda  X}]+\lambda E[e^{-\lambda  X}(\Phi_X-\sigma^2_X)]\\
&=\lambda \sigma^2_XE[e^{-\lambda  X}]-\lambda^2E\left[e^{-\lambda  X}\int_0^T D_sXE[D_s\Phi_X|\mathcal{F}_s]ds\right]
\end{align*}
Since $D_sX\geq 0$ and $D_rD_sX\geq 0,$ those imply that $\int_0^T D_sXE[D_s\Phi_X|\mathcal{F}_s]ds\geq 0,$ and hence,
$$\varphi'(\lambda)\leq \lambda \sigma^2_XE[e^{-\lambda X}]=\lambda \sigma^2_X\varphi(\lambda),\,\,\lambda>0.$$
This, together the fact $\varphi(0)=1,$ gives us
$$\varphi(\lambda)\leq e^{\frac{\lambda^2 \sigma^2_X}{2}},\,\,\lambda>0.$$
By Markov's inequality we have, for all $\lambda>0,$
$$P(X\leq x)\leq e^{\lambda x}\varphi(\lambda)\leq e^{\lambda x+\frac{\lambda^2 \sigma^2_X}{2}},\,\,x\leq 0.$$
When $x\leq 0,$ we can choose $\lambda=-\frac{x}{\sigma^2_X}$ to get
$$
P(X\leq x)\leq e^{-\frac{x^2}{2\sigma^2_X}},\,\,x\leq 0.
$$
From the estimate (\ref{kld2}), we have $\sigma^2_X=E[\Phi_X]\leq \int_0^T |D_sX|^2ds\leq \sigma^2 T^{2H}.$ So we deduce
\begin{equation}\label{jkvm1b}
P(X\leq x)\leq e^{-\frac{x^2}{2\sigma^2 T^{2H}}},\,\,x\leq 0.
\end{equation}
Combining (\ref{jkvm1a}) and (\ref{jkvm1b}) yields
$$\rho_X(x)\leq c\,e^{-\frac{x^2}{8\sigma^2 T^{2H}}},\,\,x\leq 0.
$$
where $c$ is a positive constant. Recalling $X=\ln F-E[\ln F],$ the density of $F$ satisfies $\rho_F(x)=\frac{1}{x}\rho_X(\ln x-E[\ln F]).$ When $0<x\leq e^{E[\ln F]},$ we have $y:=\ln x-E[\ln F]\leq 0.$ We thus obtain
$$\rho_F(x)=\frac{1}{x}\rho_X(y)\leq \frac{c}{x}e^{-\frac{y^2}{8\sigma^2 T^{2H}}}= \frac{c}{x}e^{-\frac{(\ln x- E[\ln F])^2}{8\sigma^2 T^{2H}}},\,\,0<x\leq e^{E[\ln F]}.$$
This completes the proof of Theorem.
\end{proof}
\begin{rem} Replacing $X$ by $F-E[F]$ in the proof of Theorem \ref{kl1}, we obtain the following Gaussian bound for the left tail
\begin{equation*}\label{jkvm1}
\rho_F(x)\leq c\,e^{-\frac{(x-E[F])^2}{8\sigma^2_F}},\,\,x\leq E[F],
\end{equation*}
where $\sigma^2_F:={\rm Var}(F)$ and $c$ is a positive constant.
\end{rem}
We now use Proposition \ref{9hj3} to estimate the right tail of the density.
\begin{thm}\label{kl2} We have
\begin{equation}\label{kfm5}
\rho_F(x)\leq \frac{c}{x}\exp\left({-\frac{(\ln x-E[\ln F])^2}{2\sigma^2T^{2H}}}\right),\,\,x>e^{E[\ln F]},
\end{equation}
where $c$ is a positive constant.
\end{thm}
\begin{proof} Let $X$ be as in Proposition \ref{kfl6}. Obviously, we have $\Phi_X\neq 0\,\,a.s.$ Moreover, from the estimates (\ref{kld2}) and (\ref{kld3}) we obtain
\begin{align*}
0\leq D_s\Phi_X&=\int_0^TD_sD_\theta XE[D_\theta X|\mathcal{F}_\theta]d\theta+\int_0^TD_\theta XE[D_sD_\theta X|\mathcal{F}_\theta]d\theta\\
&\leq 4\sigma^3\int_0^T K^2(T,\theta)K(T,s)d\theta=4\sigma^3K(T,s) T^{2H}
\end{align*}
and
$$0\leq\int_0^TD_s\Phi_XE[D_sX|\mathcal{F}_s]ds\leq 4\sigma^4 T^{2H}\int_0^TK^2(T,s)ds=4\sigma^4 T^{4H}$$
Hence, it follows from Proposition \ref{iklw} that the random variable $\frac{1}{\Phi_X^2}\int_0^TD_s\Phi_XE[D_sX|\mathcal{F}_s]ds$ belong to $L^2(\Omega).$  We also have $\frac{X}{\Phi_X}\in L^2(\Omega)$ because $-|a|T-\sigma \max\limits_{0\leq s\leq T}B^H_s+\ln T-E[\ln F]\leq X\leq |a|T+\sigma \max\limits_{0\leq s\leq T}B^H_s+\ln T-E[\ln F]$ and hence, $X\in L^p(\Omega)$ for all $p\geq 2.$

In view of Proposition \ref{9hj3}, the density $\rho_X(x)$ of $X$ is given by
\begin{equation}
\rho_X(x)= \rho_X(0)\exp\left(-\int_0^x h_X(z)dz\right)\exp\left(-\int_0^x w_X(z)dz\right),\,\,\,x\in\mathrm{supp}\,\rho_X,
\end{equation}
where $w_X(z):=E\left[\frac{X}{\Phi_X}\big| X=z\right]$ and $h_X(z):=E\left[\frac{1}{\Phi_X^2}\int_0^TD_s\Phi_XE[D_sX|\mathcal{F}_s]ds\big| X=z\right].$

Since $h_X\geq 0,$ this implies that
$$\exp\left(-\int_0^x h_X(z)dz\right)\leq 1,\,\,x\geq 0.$$
From the estimate (\ref{kld2}) we have
$$0\leq \Phi_X\leq \sigma^2\int_0^T K^2(T,\theta)d\theta=\sigma^2T^{2H}\,\,a.s.$$
and we obtain
$$\exp\left(-\int_0^x w_F(z)dz\right)\leq e^{-\frac{x^2}{2\sigma^2 T^{2H}}},\,\,x\in \mathbb{R}.$$
So we can conclude that
$$\rho_X(x)\leq  \rho_X(0)e^{-\frac{x^2}{2\sigma^2 T^{2H}}},\,\,x\geq 0,$$
and (\ref{kfm5}) follows because $\rho_F(x)=\frac{1}{x}\rho_X(\ln x-E[\ln F]).$ The proof of Theorem is complete.
\end{proof}

\noindent {\bf Acknowledgments.} The author would like to thank the anonymous referee for their valuable comments for improving the paper. This research was funded by the Vietnam National University, Hanoi under grant number QG.20.21. A part of this paper was done while the first author was visiting the Vietnam Institute for Advanced Study in Mathematics (VIASM). He would like to thank the VIASM for financial support and hospitality.

\end{document}